\documentclass{article}
\usepackage{amssymb,amsbsy}
\usepackage{amsmath}
\usepackage{amsthm}
\usepackage{mathrsfs}
\usepackage{color}
\usepackage[dvipsnames]{xcolor}
\usepackage{tikz}

\newtheorem{theorem}{Theorem}[section]
\newtheorem{lemma}[theorem]{Lemma}
\newtheorem{proposition}[theorem]{Proposition}
\newtheorem{corollary}[theorem]{Corollary}

\theoremstyle{definition}

\newtheorem{example}[theorem]{Example}
\newtheorem{algorithm}[theorem]{Algorithm}

\theoremstyle{remark}
\newtheorem{remark}[theorem]{Remark}

\begin{document}
	
	\title{Modular Frobenius pseudo-varieties}
	
	\author{Aureliano M. Robles-P\'erez\thanks{Both authors are supported by the project MTM2017-84890-P (funded by Mi\-nis\-terio de Econom\'{\i}a, Industria y Competitividad and Fondo Europeo de Desarrollo Regional FEDER) and by the Junta de Andaluc\'{\i}a Grant Number FQM-343.} \thanks{Departamento de Matem\'atica Aplicada \& Instituto de Matemáticas (IMAG), Universidad de Granada, 18071-Granada, Spain. \newline E-mail: {\bf arobles@ugr.es}. ORCID: {\bf 0000-0003-2596-1249}.}
		\mbox{ and} Jos\'e Carlos Rosales$^*$\thanks{Departamento de \'Algebra \& Instituto de Matemáticas (IMAG), Universidad de Granada, 18071-Granada, Spain. \newline E-mail: {\bf jrosales@ugr.es}. ORCID: {\bf 0000-0003-3353-4335}.} }
	
	\date{ }
	
	\maketitle
	
	\begin{abstract}
		If $m \in \mathbb{N} \setminus \{0,1\}$ and $A$ is a finite subset of $\bigcup_{k \in \mathbb{N} \setminus \{0,1\}} \{1,\ldots,m-1\}^k$, then we denote by
		\begin{align*}
			\mathscr{C}(m,A) = \Big\{S\in \mathscr{S}_m \mid s_1+\cdots+s_k-m \in S \mbox{ if } (s_1,\ldots,s_k)\in S^k \mbox{ and } \\ s_1 \bmod m, \ldots, s_k \bmod m)\in A \Big\}.
		\end{align*}
		In this work we prove that $\mathscr{C}(m,A)$ is a Frobenius pseudo-variety. We also show algorithms that allows us to establish whether a numerical semigroup belongs to $\mathscr{C}(m,A)$ and to compute all the elements of $\mathscr{C}(m,A)$ with a fixed genus. Moreover, we introduce and study three families of numerical semigroups, called of second-level, thin and strong, and corresponding to $\mathscr{C}(m,A)$ when $A=\{1,\ldots,m-1\}^3$, $A=\{(1,1),\ldots,(m-1,m-1)\}$, and $A=\{1,\ldots,m-1\}^2 \setminus \{(1,1),\ldots,(m-1,m-1)\}$, respectively.
	\end{abstract}
	
	\medskip
	
	\noindent{\it 2010 AMS Classification:} 20M14. 	
	
	\section{Introduction}
	
	Let $\mathbb{N}$ be the set of non-negative integers. A \emph{numerical semigroup} is a subset $S$ of $\mathbb{N}$ that is closed under addition, $0 \in S$ and $\mathbb{N} \setminus S$ is finite. The \emph{Frobenius number} of $S$, denoted by $\mathrm{F}(S)$, is the greatest integer that does not belong to $S$. The cardinality of $\mathbb{N}\setminus S$, denoted by $\mathrm{g}(S)$, is the \emph{genus} of $S$.
	
	A {\em Frobenius pseudo-variety} is a non-empty family $\mathcal{P}$ of numerical semigroups that fulfils the following conditions.
	\begin{enumerate}
		\setlength\itemsep{0pt}
		\item $\mathcal{P}$ has a maximum element (with respect to the inclusion order).
		\item If $S,T \in \mathcal{P}$, then $S\cap T\in \mathcal{P}$.
		\item If $ S\in \mathcal{P}$ and $S\neq \max(\mathcal{P})$, then $S\cup \{{\rm F}(S)\} \in \mathcal{P}$.
	\end{enumerate}
	
	The \emph{multiplicity} of a numerical semigroup $S$, denoted by $\mathrm{m}(S)$, is the least positive integer that belongs to $S$. If $m$ is a positive integer, then we denote by $\mathscr{S}_m=\left\{S \mid S \mbox{ is a numerical semigroup with } \mathrm{m}(S)=m \right\}$.
	
	Let $m\in \mathbb{N} \setminus \{0,1\}$ and let $A$ be a finite subset of $\bigcup_{k \in \mathbb{N} \setminus \{0,1\}} \{1,\ldots,m-1 \}^k$ (where $X^k=X\times \stackrel{^{(k)}}{\cdots} \times X = \left\{(x_1,\ldots,x_k) \mid x_1,\ldots,x_k \in X \right\}$). We denote by
	\begin{align*}
		\mathscr{C}(m,A) = \Big\{S\in \mathscr{S}_m \mid s_1+\cdots+s_k-m \in S \mbox{ if } (s_1,\ldots,s_k)\in S^k \mbox{ and } \\ s_1 \bmod m, \ldots, s_k \bmod m)\in A \Big\}.
	\end{align*}
	Our main purpose in this work is to study the set $\mathscr{C}(m,A)$.
	
	In Section~\ref{C(m,A)-pseudo-variety} we show that $\mathscr{C}(m,A)$ is a Frobenius pseudo-variety with maximum element given by $\Delta(m)=\left\{0,m,\to \right\}$ (where the symbol $\to$ means that every integer greater than $m$ belongs to $\Delta(m)$). Thus, we call the pseudo-varieties that arise in this way \textit{modular Frobenius pseudo-varieties}. Also, we give an algorithm that allows us to establish whether a numerical semigroup belongs to $\mathscr{C}(m,A)$. In addition, with the help of the results in \cite{ampa}, we can arrange $\mathscr{C}(m,A)$ in a rooted tree and we find an algorithm to compute all the elements of $\mathscr{C}(m,A)$ with a fixed genus.
	
	If $X$ is a non-empty subset of $\mathbb{N}$, then we denote by $\langle X \rangle$ the submonoid of $(\mathbb{N},+)$ generated by $X$, that is,
	$$\langle X \rangle=\big\{\lambda_1x_1+\cdots+\lambda_nx_n \mid n\in\mathbb{N}\setminus \{0\}, \ x_1,\ldots,x_n\in X, \ \lambda_1,\ldots,\lambda_n\in \mathbb{N}\big\}.$$
	It is well known (see Lemma~2.1 of \cite{springer}) that $\langle X \rangle$ is a numerical semigroup if and only if $\gcd(X)=1$. If $M$ is a submonoid of $(\mathbb{N},+)$ and $M=\langle X \rangle$, then we say that $X$ is a \emph{system of generators} of $M$. Moreover, if $M\not=\langle Y \rangle$ for any subset $Y\subsetneq X$, then we say that $X$ is a \emph{minimal system of generators} of $M$. In Corollary~2.8 of \cite{springer} it is shown that each submonoid of $(\mathbb{N},+)$ has a unique minimal system of generators and that such a system is finite. We denote by $\mathrm{msg}(M)$ the minimal system of generators of $M$. The cardinality of $\mathrm{msg}(M)$, denoted by $\mathrm{e}(M)$, is the \emph{embedding dimension} of $M$.
	
	By applying Proposition~2.10 of \cite{springer}, if $S$ is a numerical semigroup, then we know that $\mathrm{e}(S)\leq\mathrm{m}(S)$. A numerical semigroup $S$ has \emph{maximal embedding dimension} if $\mathrm{e}(S)=\mathrm{m}(S)$. This family of numerical semigroups has been extensively studied (for instance, see \cite{barucci} and \cite{springer}). Let us denote by $\mathcal{M}_m$ the set formed by the numerical semigroups that have maximal embedding dimension and with multiplicity $m$. It is easy to see that $\mathcal{M}_m=\mathscr{C}(m,\{1,\ldots,m-1\}^2)$.
	
	In Sections~\ref{2-level-ns}, \ref{thin-ns}, and \ref{strong-ns} we study the family $\mathscr{C}(m,A)$ for
	\begin{itemize}
		\setlength\itemsep{0pt}
		\item $A=\{1,\ldots,m-1\}^3$,
		\item $A=\{(1,1),\ldots,(m-1,m-1)\}$,
		\item $A=\{1,\ldots,m-1\}^2 \setminus \{(1,1),\ldots,(m-1,m-1)\}$,
	\end{itemize}
	respectively. Observe that, in a certain sense, these families are generalisation of $\mathcal{M}_m$.
	
	To finish this introduction, we are going to comment several ideas (see \cite{bags,bagsvo,klara}) that motivate the study of modular Frobenius pseudo-varieties.
	
	First of all, observe that modular Frobenius pseudo-varieties are related to the non-homogeneous patterns with positive coefficients that involve in their constant parameter the multiplicity of the numerical semigroup (see \cite{bagsvo}).
	
	The notion of non-homogeneous pattern was introduced in \cite{bagsvo} as a generalisation of the notion of homogeneous pattern \cite{bags}. Thus, a \emph{linear pattern} $p(x_1,\ldots,x_n)$ is an expression of the form $a_1x_1+\cdots+a_nx_n+a_0$, with $a_0\in\mathbb{Z}$ (as usual, $\mathbb{Z}$ is the set of integers numbers) and $a_1,\ldots,a_n\in\mathbb{Z}\setminus\{0\}$. In particular, the (linear) pattern $p$ is \emph{homogeneous} if $a_0=0$, and \emph{non-homogeneous} if $a_0\not=0$.
	
	On the other hand, it is said that a numerical semigroup $S$ admits the homogeneous pattern $p$ if $p(s_1,\ldots,s_n)\in S$ for every non-increasing sequence $(s_1,\ldots,s_n)\in S^n$. The corresponding definition for non-homogeneous patterns is a bit different: a numerical semigroup $S$ admits the non-homogeneous pattern $p$ if $p(s_1,\ldots,s_n)\in S$ for every non-increasing sequence $(s_1,\ldots,s_n)\in (S\setminus\{0\})^n$.
	
	Having in mind that $M(S)=S\setminus\{0\}$ is an ideal of the numerical semigroup $S$ (in fact, $M(S)$ is the maximal ideal of $S$), in \cite{klara} the concepts of the above paragraph were extended in the following way: an ideal $I$ of a numerical semigroups $S$ admits the pattern $p$ if $p(s_1,\ldots,s_n)\in S$ for every non-increasing sequence $(s_1,\ldots,s_n)\in I^n$.
	
	At this point we have the main difference between the above-mentioned papers and our proposal in this work: we do not keep the non-increasing condition (on the sequences in which we evaluate the pattern) in mind. In addition, we take sequences in sets $A$ without structure (that is, $A$ does not have to be an ideal of a numerical semigroup).
	
	Now, let us denote by $\mathscr{S}_m(p)$ the family of numerical semigroups with multiplicity $m$ that admit the pattern $p$. If we take the patterns $p_1=2x_1+x_2-m$ and $p_2=x_1+2x_2-m$, then $\mathscr{S}_m(p_1,p_2) = \mathscr{S}_m(p_1) \cap \mathscr{S}_m(p_2)$ (that is, the family of numerical semigroups which admit $p_1$ and $p_2$ simultaneously) is equal to $\mathscr{C}(m,A)$ for
	\[ A = \left\{ (a_1,a_2,a_3) \in \{1,\ldots,m-1\}^3 \mid a_1\equiv a_2 \pmod m \right\}, \]
	being this one an easy example of the connection between modular Frobenius pseudo-varieties and families of numerical semigroups defined by non-homogeneous patterns (as considered in \cite{bagsvo}).
	
	A first question studied in \cite{bagsvo,klara} is the next one: if $p=a_1x_1+\cdots+a_nx_n+a_0$ is a non-homogeneous pattern, for which values of $a_0$ we have that the family $\mathscr{S}(p)$, of numerical semigroups that admit the pattern $p$, is non-empty? If the answer is affirmative, then it is said that $p$ is an \emph{admissible pattern}. In our case, we have imposed that $a_0=-m$, where $m$ will be the multiplicity of all the numerical semigroups, in order to get a relevant advantage: we want to build in a (more or less) easy and explicit way the Apéry set of the numerical semigroups belonging to $\mathscr{C}(m,A)$. As a consequence of this fact, we will be able to obtain extra information about such families of numerical semigroups.
	
	Of course, in addition to $-m$, it is possible to consider other values that maintain the pseudo-variety structure. Thus, for $p=x_1+\cdots+x_n+a_0$, we have that, independently of the chosen set $A$, every numerical semigroup $S$, such that $a_0\in S$, admits the pattern $p$.
	
	Now, let us observe that, if $A_1\subseteq A_2$, then $\mathscr{C}(m,A_2)\subseteq\mathscr{C}(m,A_1)$. This fact allow us to have sufficient conditions on the patterns $p=a_1x_1+\cdots+a_nx_n-km$, with $k\in\mathbb{N}\setminus\{0,1\}$, from the results in \cite{bagsvo,klara}. For example, from \cite[Theorem 4.1]{bagsvo} or \cite[Proposition 20]{klara}, we can assert that $\mathscr{C}(m,A)\not=\emptyset$ if $n-k\geq 1$.
	
	A second question that appears in \cite{bags,klara} is about the \emph{equivalence of patterns}. Briefly, the pattern $p_1$ induces the pattern $p_2$ if $\mathscr{S}(p_2)\subseteq\mathscr{S}(p_1)$ and, moreover, $p_1$ and $p_2$ are equivalent patterns if they induce each other. Maybe the first result of this type is the equivalence between the homogeneous patterns $2x_1-x_2$ and $x_1+x_2-x_3$, which correspond to the family of Arf numerical semigroups (see \cite{CFM}). Since the set $A$ has an important role (not to say the main role) in the families $\mathscr{C}(m,A)$, the large number of possibilities in the choice of $A$ leads us to believe that this is an issue that deserves a new work.
	
	By the way, in Section~\ref{thin-ns} we study the family of thin numerical semigroups, denoted by $\mathscr{T}_m$, and in Section~\ref{strong-ns} we consider the family of strong numerical semigroups, denoted by $\mathscr{R}_m$. These families are associated with the patterns $2x-m$ and $x+y-m$, respectively. However, the set $A$ is different in each case and, as a consequence, we have that there is not a inclusion relation between both of them (see Examples~\ref{exmp28} and \ref{exmp35}). This fact may give an idea about the difficulty of obtaining similar results to those seen in \cite{bags,klara}. 
	
	In any case, we can show simple results about the equivalence (or, at least, the inclusions) of the families $\mathscr{C}(m,A)$.
	
	\begin{remark}\label{rem-01}
		All the patterns $p=x_1+\cdots+x_n-m$ are equivalent (independently of the chosen set $A$) if $n\geq m$, in which case we have that $\mathscr{C}(m,A)=\mathscr{S}_m$. Indeed, applying the pigeonhole principle, if $s_1,\ldots,s_n$ are elements of $S\in\mathscr{S}_m$ such that $s_i\not\equiv 0 \pmod m$, $1\leq i\leq n$, then there exist $i,j\in\{1,\ldots,n\}$, with $i<j$, such that $s_i+\cdots+s_j=km$ for some $k\in\mathbb{N}\setminus\{0\}$. Consequently, $S\in\mathscr{C}(m,A)$.
	\end{remark}
	
	\begin{remark}\label{rem-02}
		Let us set $A=\{(1,1),(3,4)\}$, $B=\{(1,1,2),(1,3,1),(1,3,4)\}$, and $m\geq 5$. Then $\mathscr{C}(m,A) \subseteq \mathscr{C}(m,B)$. In order to verify this inclusion, we take $a_1,a_2,a_3\in S \in \mathscr{C}(m,A)$. 
		\begin{itemize}
			\item If $a_1\equiv 1 \pmod m$, $a_2\equiv 1 \pmod m$, $a_3\equiv 2 \pmod m$, then $a_1+a_2+a_3-m=(a_1+a_2-m)+a_3\in S$ .
			\item If $a_1\equiv 1 \pmod m$, $a_2\equiv 3 \pmod m$, $a_3\equiv 1 \pmod m$, then $a_1+a_2+a_3-m=(a_1+a_3-m)+a_2\in S$.
			\item If $a_1\equiv 1 \pmod m$, $a_2\equiv 3 \pmod m$, $a_3\equiv 4 \pmod m$, then $a_1+a_2+a_3-m=(a_2+a_3-m)+a_1\in S$.
		\end{itemize}
		More generally, let us suppose that, for each $b\in B$, there exists $a\in A$ such that $b$ is obtained by adding coordinates to $a$. Then $\mathscr{C}(m,A) \subseteq \mathscr{C}(m,B)$.
	\end{remark}
	
	As a final recommendation, it is worth mentioning that in \cite[Section 2]{bagsvo} and in \cite[Introduction]{klara} there are several motivating examples, and references, as to why it is interesting to study (non-homogeneous) patterns of numerical semigroups.

	\section{The pseudo-variety $\mathscr{C}\boldsymbol{(m,A)}$}\label{C(m,A)-pseudo-variety}
	
	In this section, $m$ is an integer greater than or equal to $2$ and $A$ is a finite subset of $\bigcup_{k \in \mathbb{N} \setminus \{0,1\}} \{1,\ldots,m-1 \}^k$. Moreover, recall that
	\begin{align*}
		\mathscr{C}(m,A) = \Big\{S\in \mathscr{S}_m \mid s_1+\cdots+s_k-m \in S \mbox{ if } (s_1,\ldots,s_k)\in S^k \mbox{ and } \\ s_1 \bmod m, \ldots, s_k \bmod m)\in A \Big\},
	\end{align*}
	where $x \bmod m$ denotes the remainder after division of $x$ by $m$.
	
	If $S$ is a numerical semigroup and $x\in S\setminus \{0\}$, then the \emph{Ap\'ery set of $x$ in $S$} (see \cite{apery}) is $\mathrm{Ap}(S,x)=\{w(0)=0,w(1),\ldots,w(x-1)\}$, where $w(i)$ is the least element of $S$ that is congruent with $i$ modulus $x$. Observe that an integer $s$ belongs to $S$ if and only if there exists $t\in \mathbb{N}$ such that $s=w(s\bmod x)+tx$.
	
	\begin{proposition}\label{prop1}
		Let $S\in\mathscr{S}_m$ and $\mathrm{Ap}(S,m)=\{w(0)=0,w(1),\ldots,w(m-1)\}$. Then the following conditions are equivalent.
		\begin{enumerate}
			\setlength\itemsep{0pt}
			\item $S\in \mathscr{C}(m,A)$.
			\item $w(i_1)+\cdots+w(i_k)-m\in S$ for all $(i_1,\ldots,i_k)\in A$.
		\end{enumerate}
	\end{proposition}
	
	\begin{proof}
		\textit{(1. $\Rightarrow$ 2.)} Since $w(i_1),\ldots,w(i_k)\in S$ and $(w(i_1) \bmod m,\ldots,w(i_k) \bmod m) = (i_1,\ldots,i_k)\in A$, then $w(i_1)+\cdots+w(i_k)-m\in S$.
		
		(\textit{2. $\Rightarrow$ 1.)} Let $s_1,\ldots,s_k \in S$ such that $$(s_1 \bmod m,\ldots,s_k \bmod m) = (i_1,\ldots,i_k)\in A.$$ Then there exist $t_1,\ldots,t_k \in \mathbb{N}$ such that $s_j=w(i_j)+t_jm$, $1\leq j\leq k$, and thus, $s_1+\cdots+s_k = (w(i_1)+\cdots+w(i_k)-m)+(t_1+\cdots+t_k)m\in S$.
	\end{proof}
	
	By using the function \texttt{AperyListOfNumericalSemigroupWRTElement(S,m)} of \cite{numericalsgps}, we can compute $\mathrm{Ap}(S,m)$ from a system of generators of $S$. Thereby, we have the following algorithm to decide if a numerical semigroup $S$ belongs or not to $\mathscr{C}(m,A)$.
	
	\begin{algorithm}\label{alg2}
		\mbox{ } \par	
		\noindent $\ $ INPUT: A finite subset $G$ of positive integers. \par
		\noindent $\ $ OUTPUT: $\langle G \rangle \in \mathscr{C}(m,A)$ or $\langle G \rangle \notin \mathscr{C}(m,A)$.
		\vspace{-3pt}
		\begin{itemize}
			\setlength\itemsep{0pt}
			\item[(1)] If $\min$(G)$\not=m$, return $\langle G \rangle \notin \mathscr{C}(m,A)$.
			\item[(2)] If $\gcd(G)\not=1$, return $\langle G \rangle \notin \mathscr{C}(m,A)$.
			\item[(3)] Compute $\mathrm{Ap}(\langle G \rangle, m) = \{w(0),w(1),\ldots,w(m-1)\}$.
			\item[(4)] If $w(i_1)+\cdots+w(i_k)-m\in S$ for all $(i_1,\ldots,i_k)\in A$, return $\langle G \rangle \in \mathscr{C}(m,A)$.
			\item[(5)] Return $\langle G \rangle \notin \mathscr{C}(m,A)$.
		\end{itemize}
	\end{algorithm}
	
	Let us illustrate the working of the previous algorithm through an example.
	
	\begin{example}\label{exmp3}
		Let us make use of Algorithm~\ref{alg2} with $G=\{5,7,9\}$, $m=5$, and $A=\{(1,3),(2,2)\}$.
		\begin{itemize}
			\setlength\itemsep{0pt}
			\item $\min(G)=5$.
			\item $\gcd(G)=1$.
			\item $\mathrm{Ap}(\langle G \rangle,5) = \{w(0)=0, w(1)=16, w(2)=7, w(3)=18, w(4)=19\}$.
			\item $w(1)+w(3)-5=29\in \langle G \rangle$ and $w(2)+w(2)-5=9\in \langle G \rangle$.	
			\item[] Therefore, $\langle G \rangle = \langle 5,7,9 \rangle \in \mathscr{C}(5,\{(1,3),(2,2)\})$.
		\end{itemize}
	\end{example}
	
	Recall that, if $m$ is an integer greater than or equal to $2$, then we denote by $\Delta(m)=\{0,m,\to\}$. It is clear that $\Delta(m)\in\mathscr{S}_m$ and that, if $s_1,\ldots,s_k\in \Delta(m)\setminus\{0\}$ and $k\geq2$, then $s_1+\cdots+s_k-m\in \Delta(m)$. Therefore, we have the next result.
	
	\begin{lemma}\label{lem4}
		Let $m\in \mathbb{N}\setminus\{0,1\}$ and $A\subseteq \bigcup_{k\in\{2,\to\}} \{1,\ldots,m-1\}^k$ ($A$ finite). Then $\Delta(m)\in \mathscr{C}(m,A)$.
	\end{lemma}
	
	Let $m\in \mathbb{N}\setminus\{0,1\}$. Then it is easy to show that $S\cap T\in \mathscr{S}_m$ for all $S,T\in\mathscr{S}_m$. Moreover, $\Delta(m)$ is the maximum of $\mathscr{S}_m$ and, if $S\in \mathscr{S}_m$ and $S\not=\Delta(m)$, then $S\cup \{\mathrm{F}(S)\} \in \mathscr{S}_m$. From all this, we conclude the following result.
	
	\begin{lemma}\label{lem5}
		Let $m\in \mathbb{N}\setminus\{0,1\}$. Then $\mathscr{S}_m$ is a Frobenius pseudo-variety with $\Delta(m)$ as maximum element.
	\end{lemma}
	
	\begin{proposition}\label{prop6}
		$\mathscr{C}(m,A)$ is a Frobenius pseudo-variety.
	\end{proposition}
	
	\begin{proof}
		From Lemmas~\ref{lem4} and \ref{lem5}, we have that $\Delta(m)$ is the maximum of $\mathscr{C}(m,A)$. It is also easy to see that, if $S,T \in \mathscr{C}(m,A)$, then $S\cap T\in \mathscr{C}(m,A)$. In order to finish the proof, let us see that, if $S\in \mathscr{C}(m,A)$ and $S\not=\Delta(m)$, then $S\cup \{\mathrm{F}(S)\} \in \mathscr{C}(m,A)$. For that, we have to show that, if $s_1,\ldots,s_k\in S\cup \{\mathrm{F}(S)\}$ and $(s_1 \bmod m, \ldots, s_k \bmod m) \in A$, then $s_1+\cdots+s_k-m\in S\cup  \{\mathrm{F}(S)\}$. In effect, if $\mathrm{F}(S)\notin\{s_1,\ldots,s_k\}$, then the result is true because $S\in \mathscr{C}(m,A)$. On the other hand, if $\mathrm{F}(S)\in\{s_1,\ldots,s_k\}$, then $s_1+\cdots+s_k-m\geq \mathrm{F}(S)$ and, consequently, $s_1+\cdots+s_k-m \in S\cup  \{\mathrm{F}(S)\}$.
	\end{proof}
	
	Our next purpose in this section is to show an algorithm that allows us to build all the elements of $\mathscr{C}(m,A)$ that have a fixed genus. To do that, we use the concept of rooted tree.
	
	A \emph{graph} $G$ is a pair $(V,E)$ where $V$ is a non-empty set (whose elements are called \emph{vertices} of $G$) and $E$ is a subset of $\{(v,w) \in V \times V \mid v \neq w\}$ (whose elements are called \emph{edges} of $G$). A \emph{path (of length $n$) connecting the vertices $x$ and $y$ of $G$} is a sequence of different edges $(v_0,v_1),(v_1,v_2),\ldots,(v_{n-1},v_n)$ such that $v_0=x$ and $v_n=y$.
	
	We say that a graph $G$ is a \emph{(rooted) tree} if there exists a vertex $r$ (known as the \emph{root} of $G$) such that, for any other vertex $x$ of $G$, there exists a unique path connecting $x$ and $r$. If there exists a path connecting $x$ and $y$, then we say that $x$ is a \emph{descendant} of $y$. In particular, if $(x,y)$ is an edge of the tree, then we say that $x$ is a \emph{child} of $y$. (See \cite{rosen}.)
	
	We define the graph $\mathrm{G}\big(\mathscr{C}(m,A)\big)$ in the following way: $\mathscr{C}(m,A)$ is the set of vertices and $(S,S')\in \mathscr{C}(m,A) \times \mathscr{C}(m,A)$ is an edge if $S\cup \{\mathrm{F}(S)\}=S'$. The following result is a consequence of Lemma~4.2 and Theorem~4.3 of \cite{ampa}.
	
	\begin{theorem}\label{thm7}
		$\mathrm{G}\big(\mathscr{C}(m,A)\big)$ is a tree with root $\Delta(m)$. Moreover, the children set of $S\in \mathscr{C}(m,A)$ is $\big\{S \setminus \{x\} \mid x \in \mathrm{msg}(S), \ S\setminus\{x\} \in \mathscr{C}(m,A), \ x>\mathrm{F}(S) \big\}.$
	\end{theorem}
	
	Let $S$ be a numerical semigroup and let $x\in S$. Then it is clear that $S\setminus \{x\}$ is a numerical semigroup if and only if $x\in\mathrm{msg}(S)$. Moreover, let us observe that, if $x\in\mathrm{msg}(S)$ and $m\in S\setminus\{0,x\}$, then
	\[\mathrm{Ap}(S\setminus \{x\},m) = \big(\mathrm{Ap}(S,m)\setminus\{x\}\big)\cup\{x+m\}.\]
	
	In the following result we characterise the children of $S\in \mathscr{C}(m,A)$.
	
	\begin{proposition}\label{prop8}
		Let $S\in \mathscr{C}(m,A)$, $\mathrm{Ap}(S,m)=\{w(0),w(1),\ldots,w(m-1)\}$ and $x\in\mathrm{msg}(S)\setminus\{m\}$. Then $S\setminus\{x\}\in \mathscr{C}(m,A)$ if and only if $w(i_1)+\cdots+w(i_k)\neq m+x$ for all $(i_1,\ldots,i_k)\in A$.
	\end{proposition}
	
	\begin{proof}
		Let us suppose that $\mathrm{Ap}(S\setminus \{x\},m) = \big(\mathrm{Ap}(S,m)\setminus\{x\}\big)\cup\{x+m\}=\{w'(0),w'(1),\ldots,w'(m-1)\}$.
		
		\textit{(Necessity.)} If $w(i_1)+\cdots+w(i_k) = m+x$, then $m+x \notin \{w(i_1),\ldots,w(i_k)\}$ and, therefore, $w(i_1)=w'(i_1),\ldots,w(i_k)=w'(i_k)$. Moreover, $w(i_1)+\cdots+w(i_k)-m=x\notin S\setminus\{x\}$ and, consequently, $S\setminus\{x\} \notin \mathscr{C}(m,A)$.
		
		\textit{(Sufficiency.)} In order to prove that $S\setminus\{x\}\in\mathscr{C}(m,A)$, by Proposition~\ref{prop1}, it is enough to see that, if $(i_1,\ldots,i_k)\in A$, then $w'(i_1)+\cdots+w'(i_k)-m\in S\setminus\{x\}$.
		
		Since $S\in \mathscr{C}(m,A)$, we easily deduce that $w'(i_1)+\cdots+w'(i_k)-m\in S$. Now, if $w'(i_1)+\cdots+w'(i_k)-m\notin S\setminus\{x\}$, then $w'(i_1)+\cdots+w'(i_k)=m+x$. Thus,  $w(i_1)=w'(i_1),\ldots,w(i_k)=w'(i_k)$ and $w(i_1)+\cdots+w(i_k)=m+x$, where the last equality is in contradiction with $S\in \mathscr{C}(m,A)$.
	\end{proof}
	
	Let us observe that a tree can be built in a recurrent way starting from its root and connecting each vertex with its children. Let us also observe that the elements of $\mathscr{C}(m,A)$ with genus equal to $g+1$ are precisely the children of the elements of $\mathscr{C}(m,A)$ with genus equal to $g$.
	
	We are ready to show the above announced algorithm.
	
	\begin{algorithm}\label{alg9}
		\mbox{ } \par	
		\noindent $\ $ INPUT: A positive integer $g$. \par
		\noindent $\ $ OUTPUT: $\{S \in \mathscr{C}(m,A) \mid \mathrm{g}(S)=g \}$.
		\vspace{-3pt}
		\begin{itemize}
			\setlength\itemsep{0pt}
			\item[(1)] If $g<m-1$, return $\emptyset$.
			\item[(2)] $X=\{\Delta(m)\}$ and $i=m-1$.
			\item[(3)] If $i=g$, return $X$.
			\item[(4)] For each $S\in X$, compute the set \[\mathscr{B}_S=\big\{x\in\mathrm{msg}(S) \mid x>\mathrm{F}(S), \ x\not=m, \ S\setminus\{x\}\in \mathscr{C}(m,A) \big\}.\]
			\item[(5)] If $\ \bigcup_{S\in X} \mathscr{B}_S=\emptyset$, return $\emptyset$.
			\item[(6)] $X=\bigcup_{S\in X} \big\{S\setminus\{x\} \mid x \in \mathscr{B}_S \big\}$, $i=i+1$, and go to (3).
		\end{itemize}
	\end{algorithm}
	
	For computing item (4) in the above algorithm, we use Proposition~\ref{prop8} and the following Lemma~\ref{lem10}, which is a reformulation of Corollary~18 of \cite{frases}, and is useful to obtain the minimal system of generators of $S\setminus\{x\}$ starting from the minimal system of generators of $S$.
	
	\begin{lemma}\label{lem10}
		Let $S$ be a numerical semigroup with $\mathrm{msg}(S)=\{n_1<n_2<\cdots<n_e\}$. If $i \in \{2,\ldots,e\}$ and $n_i>\mathrm{F}(S)$, then
		\[ \mathrm{msg}(S \setminus \{n_i\})= \left\{ \begin{array}{l}
			\{n_1,\ldots,n_{e}\} \setminus \{n_i\}, \quad \mbox{if there exists } j \in \{2,\ldots,i-1\} \\ \hspace{3.4cm} \mbox{such that } n_i+n_1-n_j \in S; \\[2mm]
			\big(\{n_1,\ldots,n_{e}\} \setminus \{n_i\}\big) \cup \{n_i+n_1\}, \quad \mbox{in other case.}
		\end{array} \right.\]
	\end{lemma}
	
	We illustrate the functioning of Algorithm~\ref{alg9} with an example.
	
	\begin{example}\label{exmp11}
		Let us compute all the elements of $\mathscr{C}(5,\{(1,1),(1,2)\})$ with genus equal to $6$.
		\begin{itemize}
			\setlength\itemsep{0pt}
			\item $X=\{\langle 5,6,7,8,9 \rangle\}$ and $i=4$.
			\item $B_{\langle 5,6,7,8,9 \rangle} = \{6,9\}$.
			\item $X=\{\langle 5,7,8,9,11 \rangle, \langle 5,6,7,8 \rangle\}$ and $i=5$.
			\item $B_{\langle 5,7,8,9,11 \rangle} = \{7,8,9,11\}$ and $B_{\langle 5,6,7,8 \rangle} = \emptyset$.
			\item $X=\{\langle 5,8,9,11,12 \rangle, \langle 5,7,9,11,13\rangle, \langle 5,7,8,11\rangle, \langle 5,7,8,9\rangle \}$ and $i=6$.
			\item Return $\{\langle 5,8,9,11,12 \rangle, \langle 5,7,9,11,13\rangle, \langle 5,7,8,11\rangle, \langle 5,7,8,9\rangle \}$.
		\end{itemize}
	\end{example}
	
	Taking advantage of the above example, we finish this section building the first three levels of the tree $\mathrm{G}\big(\mathscr{C}(5,\{(1,1),(1,2)\})\big)$.
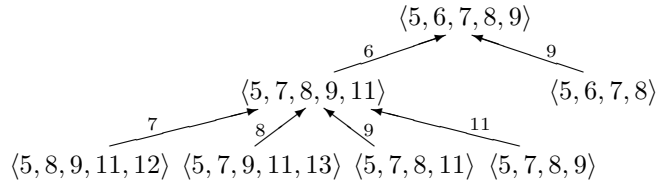
	\begin{center}
		\begin{picture}(258,65)
			
			\put(154,56){$\langle 5,6,7,8,9 \rangle$}
			
			\put(141,44){\scriptsize 6} \put(210,44){\scriptsize 9}
			\put(130,38){\vector(3,1){42}} \put(224,38){\vector(-3,1){42}}
			\put(94,28){$\langle 5,7,8,9,11 \rangle$} \put(211,28){$\langle 5,6,7,8 \rangle$}
			
			\put(59,16){\scriptsize 7} \put(100,14){\scriptsize 8} \put(141,14){\scriptsize 9} \put(181,16){\scriptsize 11}
			\put(45,10){\vector(4,1){56}} \put(100,10){\vector(4,3){19}} \put(145,10){\vector(-4,3){19}} \put(200,10){\vector(-4,1){56}}
			\put(7,0){$\langle 5,8,9,11,12 \rangle$} \put(72,0){$\langle 5,7,9,11,13 \rangle$} \put(137,0){$\langle 5,7,8,11 \rangle$} \put(188,0){$\langle 5,7,8,9 \rangle$}
		\end{picture}
	\end{center}
	Observe that the number which appear next to each edge $(S',S)$ is the minimal generator $x$ of $S$ such that $S'=S \setminus \{x\}$.

	\section{Second-level numerical semigroups}\label{2-level-ns}
	
	We say that a numerical semigroup $S$ is of \textit{second-level} if $x+y+z-\mathrm{m}(S)\in S$ for all $(x,y,z)\in (S\setminus\{0\})^3$. We denote by $\mathscr{L}_{2,m}$ the set of all the second-level numerical semigroups with multiplicity equal to $m$.
	
	\begin{proposition}\label{prop12}
		Let $S$ be a numerical semigroup with minimal system of generators given by $\{m=n_1<n_2<\cdots<n_e\}$. Then the following two conditions are equivalents.
		\begin{enumerate}
			\item $S\in \mathscr{L}_{2,m}$.
			\item If $(i,j,k)\in \{2,\ldots,e\}^3$, then $n_i+n_j+n_k-m\in S$.
		\end{enumerate}
	\end{proposition}
	
	\begin{proof}
		\textit{(1. $\Rightarrow$ 2.)} It is evident from the definition of second-level numerical semigroup.
		
		\textit{(2. $\Rightarrow$ 1.)} Let $(x,y,z) \in (S\setminus\{0\})^3$. If $0\in\{x\bmod m, y\bmod m, z\bmod m  \}$, then it is clear that $x+y+z-m\in S$. Now, if $0\notin\{x\bmod m, y\bmod m, z\bmod m  \}$, then we easily deduce that there exist $(i,j,k)\in \{2,\ldots,e\}^3$ and $(s_1,s_2,s_3)\in S^3$ such that $(x,y,z)=(n_i,n_j,n_k)+(s_1,s_2,s_3)$. Therefore, $x+y+z-m=(n_i+n_j+n_k-m)+s_1+s_2+s_3\in S$. Consequently, $S\in \mathscr{L}_{2,m}$.
	\end{proof}
	
	The above proposition allows us to easily decide whether a numerical semigroup is of second-level or not.
	
	\begin{example}\label{exmp13}
		Let us see that $S=\langle 5,7,16 \rangle \in \mathscr{L}_{2,m}$. In effect, it is clear that $\{7+7+7-5,7+7+16-5,7+16+16,16+16+16-5\} = \{16,25,34,43\} \subseteq S$. Therefore, by Proposition~\ref{prop12}, we have that $S\in \mathscr{L}_{2,m}$.
	\end{example}
	
	Now our intention is to show that $\mathscr{L}_{2,m}$ is a modular Frobenius pseudo-variety.
	
	\begin{proposition}\label{prop14}
		Let $m\in \mathbb{N}\setminus\{0,1\}$. Then $\mathscr{L}_{2,m}=\mathscr{C}(m,\{1,\ldots,m-1\}^3)$.
	\end{proposition}
	
	\begin{proof}
		Let $S\in\mathscr{L}_{2,m}$ and $\mathrm{Ap}(S,m)=\{w(0),w(1),\ldots,w(m-1)\}$. If $(i,j,k)\in\{1,\ldots,m-1\}^3$, then $(w(i),w(j),w(k))\in(S\setminus\{0\})^3$ and, therefore, $w(i)+w(j)+w(k)-m\in S$. By applying Proposition~\ref{prop1}, we have that $S\in\mathscr{C}(m,\{1,\ldots,m-1\}^3)$.
		
		In order to see the other inclusion, let $S\in \mathscr{C}(m,\{1,\ldots,m-1\}^3)$ and $(x,y,z)\in(S\setminus\{0\})^3$. Firstly, if $0\in\{x\bmod m,y\bmod m,z\bmod m\}$, then it is clear that $x+y+z-m\in S$. Secondly, if $0\notin\{x\bmod m, y\bmod m, z\bmod m  \}$, then there exist $(i,j,k)\in \{1,\ldots,m-1\}^3$ and $(p,q,r)\in \mathbb{N}^3$ such that $x=w(i)+pm$, $y=w(j)+qm$, and $z=w(k)+rm$. Thus, $x+y+z-m=(w(i)+w(j)+w(k)-m)+(p+q+r)m\in S$ and, therefore, $S\in \mathscr{L}_{2,m}$.
	\end{proof}
	
	As an immediate consequence of Propositions~\ref{prop6} and \ref{prop14}, we have the following result.
	
	\begin{corollary}\label{cor15}
		Let $m\in \mathbb{N}\setminus\{0,1\}$. Then $\mathscr{L}_{2,m}$ is a modular pseudo-Frobenius variety and, in addition, $\Delta(m)$ is the maximum of $\mathscr{L}_{2,m}$.
	\end{corollary}
	
	Our next step is to build the tree associated with the pseudo-variety $\mathscr{L}_{2,m}$. To do this, we should characterise the possible children of each element in $\mathscr{L}_{2,m}$.
	
	Let $S$ be a numerical semigroup with $\mathrm{msg}(S)=\{n_1,\ldots,n_e\}$. If $s\in S$, then we denote by
	\[L_S(s)=\max\{a_1+\cdots+a_e \mid (a_1,\ldots,a_e)\in\mathbb{N}^e \mbox{ and } a_1n_1+\cdots+a_en_e=s\}.\]
	
	\begin{proposition}\label{prop16}
		Let $m\in \mathbb{N}\setminus\{0,1\}$, $S\in \mathscr{L}_{2,m}$, and $x\in\mathrm{msg}(S)\setminus\{m\}$. Then $S\setminus\{x\}\in\mathscr{L}_{2,m}$ if and only if $L_{S\setminus\{x\}}(x+m)\leq2$.
	\end{proposition}
	
	\begin{proof}	
		\textit{(Necessity.)} Let us suppose that $L_{S\setminus\{x\}}(x+m)\geq3$. Then there exists $(a,b,c)\in(S\setminus\{0,x\})^3$ such that $x+m=a+b+c$. Thus, $a+b+c-m=x\notin S\setminus\{x\}$ and, therefore, $S\setminus\{x\}\notin\mathscr{L}_{2,m}$.
		
		\textit{(Sufficiency.)} If $(a,b,c)\in(S\setminus\{0,x\})^3$, then $a+b+c-m\in S$, since $S\in \mathscr{L}_{2,m}$. Now, if $a+b+c-m=x$, then $L_{S\setminus\{x\}}(x+m)\geq3$. Therefore, $a+b+c-m\not= x$ and, consequently, $a+b+c-m\in S\setminus\{x\}$. Thus, we conclude that $S\setminus\{x\}\in \mathscr{L}_{2,m}$.
	\end{proof}
	
	By applying Theorem~\ref{thm7}, Propositions~\ref{prop14} and \ref{prop16}, and Lemma~\ref{lem10}, we can easily build the tree $\mathrm{G}(\mathscr{L}_{2,m})$.
	
	\begin{example}\label{exmp17}
		The first four levels of $\mathrm{G}(\mathscr{L}_{2,4})$ appear in the following figure.
		
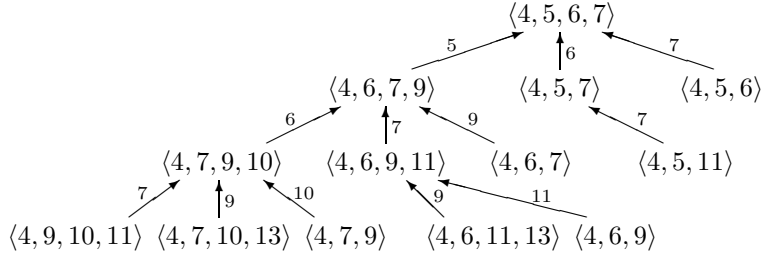
\begin{center}
			\begin{picture}(298,93)
				
				\put(196,84){$\langle 4,5,6,7 \rangle$}
				
				\put(173,73){\scriptsize 5} \put(218,70){\scriptsize 6} \put(257,73){\scriptsize 7}
				\put(160,66){\vector(3,1){42}} \put(216,66){\vector(0,1){14}} \put(274,66){\vector(-3,1){42}}
				\put(128,56){$\langle 4,6,7,9 \rangle$} \put(200,56){$\langle 4,5,7 \rangle$} \put(261,56){$\langle 4,5,6 \rangle$}
				
				\put(112,45){\scriptsize 6} \put(152,42){\scriptsize 7} \put(181,45){\scriptsize 9} \put(245,45){\scriptsize 7}
				\put(105,38){\vector(2,1){28}} \put(150,38){\vector(0,1){14}} \put(191,38){\vector(-2,1){28}} \put(255,38){\vector(-2,1){28}}
				\put(65,28){$\langle 4,7,9,10 \rangle$} \put(127,28){$\langle 4,6,9,11 \rangle$} \put(189,28){$\langle 4,6,7 \rangle$} \put(245,28){$\langle 4,5,11 \rangle$}
				
				\put(56,17){\scriptsize 7} \put(89,14){\scriptsize 9} \put(115,17){\scriptsize 10} \put(168,16){\scriptsize 9} \put(205,16){\scriptsize 11}
				\put(53,10){\vector(4,3){19}} \put(87,10){\vector(0,1){14}} \put(123,10){\vector(-4,3){19}} \put(172,10){\vector(-1,1){14}} \put(226,10){\vector(-4,1){56}}
				\put(7,0){$\langle 4,9,10,11 \rangle$} \put(63,0){$\langle 4,7,10,13 \rangle$} \put(119,0){$\langle 4,7,9 \rangle$}
				\put(165,0){$\langle 4,6,11,13 \rangle$} \put(221,0){$\langle 4,6,9 \rangle$}
			\end{picture}
		\end{center}
	\end{example}
	
	The Frobenius problem (see \cite{alfonsin}) consists in finding formulas that allow us to compute the Frobenius number and the genus of a numerical semigroup in terms of the minimal system of generators of such a numerical semigroup. This problem was solved in \cite{sylvester} for numerical semigroups with embedding dimension two. At present, the Frobenius problem is open for embedding dimension greater than or equal to $3$. However, if we know the Ap\'ery set $\mathrm{Ap}(S,x)$ for some $x\in S\setminus\{0\}$, then we have solved the Frobenius problem for $S$ because we have the following result from \cite{selmer}.
	
	\begin{lemma}\label{lem18}
		Let $S$ be a numerical semigroup and let $x\in S\setminus\{0\}$. Then
		\begin{enumerate}
			\item $\mathrm{F}(S)=(\max(\mathrm{Ap}(S,x)))-x$,
			\item $\mathrm{g}(S)=\frac{1}{x}(\sum_{w\in \mathrm{Ap}(S,x)} w)-\frac{x-1}{2}$.
		\end{enumerate}
	\end{lemma}
	
	The knowledge of $\mathrm{Ap}(S,x)=\{w(0),w(1),\ldots,w(x-1)\}$ also allows us to determine the membership of an integer to the numerical semigroup $S$. In fact, if $n\in\mathbb{N}$, then $n\in S$ if and only if $n\geq w(n\bmod x)$.
	
	Now our purpose is to show that, if $S\in \mathscr{L}_{2,m}$, then is rather easy to compute $\mathrm{Ap}(S,m)$. We need the following easy result.
	
	\begin{lemma}\label{lem19}
		Let $m\in\mathbb{N}\setminus\{0,1\}$, $S\in\mathscr{L}_{2,m}$, $\mathrm{msg}(S)=\{n_1\!=\!m,n_2,\ldots,n_e\}$. Then $\{0,n_2,\ldots,n_e\} \! \subseteq \! \mathrm{Ap}(S,m) \! \subseteq \! \{0,n_2,\ldots,n_e\} \cup \{n_i+n_j \mid (i,j)\in \{2,\ldots,e\}^2\}.$
	\end{lemma}
	
	As an immediate consequence of the above lemma we can formulate the following result.
	
	\begin{proposition}\label{prop20}
		Let $m\in\mathbb{N}\setminus\{0,1\}$, $S\in\mathscr{L}_{2,m}$, $\mathrm{msg}(S)=\{n_1\!=\!m,n_2,\ldots,n_e\}$. Then $\mathrm{Ap}(S,m)=\{w(0),w(1),\ldots,w(m-1)\}$ where $w(i)$ is the least element of $\{0,n_2,\ldots,n_e\} \cup \{n_i+n_j \mid (i,j)\in\{2,\ldots,e\}^2\}$ that is congruent to $i$ modulo $m$.
	\end{proposition}
	
	\begin{corollary}
		Let $m\in\mathbb{N}\setminus\{0,1\}$ and $S\in\mathscr{L}_{2,m}$. Then $m = \mathrm{m}(S) \leq \frac{\mathrm{e}(S)\left(\mathrm{e}(S)+1\right)}{2}$.
	\end{corollary}
	
	Following the notation introduced in \cite{JPAA}, we say that $x\in \mathbb{Z}\setminus S$ is a \textit{pseudo-Frobenius number} of $S$ if $x+s\in S$ for all $s\in S\setminus\{0\}$. We denote by $\mathrm{PF}(S)$ the set of all the pseudo-Frobenius numbers of $S$. The cardinality of $\mathrm{PF}(S)$ is an important invariant of $S$ (see \cite{barucci}) that is the so-called
	\textit{type} of $S$ and it is denoted by $\mathrm{t}(S)$.
	
	Let $S$ be a numerical semigroup. Then we define the following binary relation over $\mathbb{Z}$: $a\leq_S b$ if $b-a\in S$. In \cite{springer} it is shown that $\leq_S$ is a partial order (that is, reflexive, transitive, and antisymmetric). Moreover, Proposition~2.20 of \cite{springer} is the next result.
	
	\begin{proposition}\label{prop21}
		Let $S$ be a numerical semigroup and $x\in S\setminus\{0\}$. Then
		\[\mathrm{PF}(S)=\{w-x \mid w\in \mathrm{Maximals}_{\leq_S} (\mathrm{Ap}(S,x))  \}\]
	\end{proposition}
	
	Let us observe that, if $w,w'\in\mathrm{Ap}(S,x)$, then $w'-w\in S$ if and only if $w'-w\in \mathrm{Ap}(S,x)$. Therefore, $\mathrm{Maximals}_{\leq_S} (\mathrm{Ap}(S,x))$ is the set
	\[\{w\in\mathrm{Ap}(S,x) \mid w'-w\notin\mathrm{Ap}(S,x)\setminus\{0\} \mbox{ for all } w'\in\mathrm{Ap}(S,x) \}.\]
	
	We finish this section with an example that illustrates the above results.
	
	\begin{example}\label{exmp22}
		Having in mind that $S=\langle 5,6,13 \rangle$ is a second-level numerical semigroup, it is easy to compute $\mathrm{Ap}(S,5)$. In fact, by applying Lemma~\ref{lem19}, we have that $\mathrm{Ap}(S,5) \subseteq \{0,6,13\} \cup \{12,19,26\}$ and, by Proposition~\ref{prop20}, we conclude that $\mathrm{Ap}(S,5) = \{0,6,12,13,19\}$. On the other hand, by Lemma~\ref{lem18}, we know that $\mathrm{F}(S)=19-5=14$ and $\mathrm{g}(S)=\frac{1}{5}(6+12+13+19)-\frac{5-1}{2}=8$. Finally, since $\mathrm{Maximals}_{\leq_S} (\mathrm{Ap}(S,5)) = \{12,19\}$, Proposition~\ref{prop21} allows us to claim that $\mathrm{PF}(S)=\{7,12\}$ and $\mathrm{t}(S)=2$.
	\end{example}
	
	\begin{remark}\label {rem-03}
		The definition of second-level numerical semigroup can be easily generalised to greater levels. Thus, we say that a numerical semigroup is of \textit{$n$th-level} if $x_1+\cdots+x_{n+1}-\mathrm{m}(S)\in S$ for all $(x_1,\ldots,x_{n+1})\in (S\setminus\{0\})^{n+1}$ and denote by $\mathscr{L}_{n,m}$ the set of all the $n$th-level numerical semigroups with multiplicity equal to $m$.
		
		It is clear that for $n$th-level we obtain similar results to those of second-level. In particular, $\mathscr{L}_{n,m}=\mathscr{C}(m,\{1,\ldots,m-1\}^{n+1})$ and Proposition~\ref{prop16} remains true taking $L_{S\setminus\{x\}}(x+m)\leq n$.
		
		On the other hand, having in mind that $\mathscr{L}_{1,m}$ is the family of numerical semigroups with maximal embedding dimension and that (by Remark~\ref{rem-01}) $\mathscr{L}_{m-1,m}=\mathscr{S}_m$, we can observe that
		\[ \mathscr{L}_{1,m} \subseteq \mathscr{L}_{2,m} \subseteq \cdots \subseteq \mathscr{L}_{m-1,m} = \mathscr{S}_m, \]
		where all the inclusions are strict, as we can deduce from the following example.
	\end{remark}
	
	\begin{example}\label{exmp-01}
		Let us set $m=5$ and
		\[ S_1=\langle 5,6 \rangle =\{0,5,6,10,11,12,15,16,17,18,20, \to\}. \]
		Then,
		\begin{itemize}
			\item $S_1\in\mathscr{L}_{4,5} \setminus \mathscr{L}_{3,5} = \mathscr{S}_5 \setminus \mathscr{L}_{3,5}$, because $4\times6-5\not\in S_1$ and $s_1+\cdots+s_5-5 \geq 5\times6-5\geq F(S_1)+1$;
			\item $S_2=S_1\cup\{19\}=\langle 5,6,19 \rangle\in\mathscr{L}_{3,5} \setminus \mathscr{L}_{2,5}$, because $3\times6-5\not\in S_2$ and $s_1+\cdots+s_4-5 \geq 4\times6-5\geq F(S_2)+1$;
			\item $S_3=S_2\cup\{13,14\}=\langle 5,6,13,14 \rangle\in\mathscr{L}_{2,5} \setminus \mathscr{L}_{1,5}$, because $2\times6-5\not\in S_2$ and $s_1+s_2+s_3-5 \geq 3\times6-5\geq F(S_3)+1$.
		\end{itemize}
		Generalising this example to other values of $m$ is trivial.
	\end{example}
	
	\begin{remark}
		Let us observe that the chain obtained in Remark~\ref{rem-03} is reminiscent, in some sense, of the chain associated with subtraction patterns (see \cite[Section 6]{bags}).
	\end{remark}

	\section{Thin numerical semigroups}\label{thin-ns}
	
	We say that a numerical semigroup $S$ is \textit{thin} if $2x-\mathrm{m}(S)\in S$ for all $x\in S\setminus\{0\}$. We denote by $\mathscr{T}_{m}$ the set of all the thin numerical semigroups with multiplicity equal to $m$.
	
	\begin{proposition}\label{prop23}
		Let $S$ be a numerical semigroup with minimal system of generators given by $\{m=n_1<n_2<\cdots<n_e\}$. Then the following two conditions are equivalents.
		\begin{enumerate}
			\item $S\in \mathscr{T}_{m}$.
			\item $2n_i-m\in S$ for all $i\in \{2,\ldots,e\}$.
		\end{enumerate}
	\end{proposition}
	
	\begin{proof}
		\textit{(1. $\Rightarrow$ 2.)} It follows from the definition of thin numerical semigroup.
		
		\textit{(2. $\Rightarrow$ 1.)} Let $x \in S\setminus\{0\}$. If $x\equiv 0\pmod m$, then it is clear that $2x-m\in S$. Now, if $x\not\equiv 0\pmod m$, then we have that there exist $i\in \{2,\ldots,e\}$ and $s\in S$ such that $x=n_i+s$. Therefore, $2x-m=(2n_i-m)+2s\in S$ and, consequently, $S\in \mathscr{T}_{m}$.
	\end{proof}
	
	The above proposition allows us to easily decide whether a numerical semigroup is thin or not.
	
	\begin{example}\label{exmp24}
		Let us see that $S=\langle 4,6,7 \rangle \in \mathscr{T}_{4}$. In effect, it is clear that $\{2\cdot6-4,2\cdot7-4\} = \{8,10\} \subseteq S$. Therefore, by Proposition~\ref{prop23}, we have that $S\in \mathscr{T}_{m}$.
	\end{example}
	
	Now we want to show that $\mathscr{T}_{m}$ is a modular Frobenius pseudo-variety.
	
	\begin{proposition}\label{prop25}
		$\mathscr{T}_{m}=\mathscr{C}(m,\{(1,1),(2,2)\ldots,(m-1,m-1)\})$ for all $m\in \mathbb{N}\setminus\{0,1\}$.
	\end{proposition}
	
	\begin{proof}
		If $S\in\mathscr{T}_{m}$ and $\mathrm{Ap}(S,m)=\{w(0),w(1),\ldots,w(m-1)\}$, then it is clear that $\{w(1)+w(1)-m,\ldots,w(m-1)+w(m-1)-m\}\subseteq S$. Therefore, by applying Proposition~\ref{prop1}, we have that $S\in\mathscr{C}(m,\{(1,1),(2,2)\ldots,(m-1,m-1)\})$.
		
		To see the other inclusion, let $S\in\mathscr{C}(m,\{(1,1),(2,2)\ldots, (m-1,m-1)\})$ and $x\in S\setminus\{0\}$. On the one hand, if $x\equiv 0 \pmod m$, then it is clear that $2x-m\in S$. On the other hand, if $x\not\equiv 0 \pmod m$, then there exist $i\in \{1,\ldots,m-1\}$ and $t\in \mathbb{N}$ such that $x=w(i)+tm$. Therefore, $2x-m=(w(i)+w(i)-m)+ 2tm \in S$ and in consequence $S\in\mathscr{T}_{m}$.
	\end{proof}
	
	From Propositions~\ref{prop6} and \ref{prop25}, we get the following result.
	
	\begin{corollary}\label{cor26}
		Let $m\in \mathbb{N}\setminus\{0,1\}$. Then $\mathscr{T}_{m}$ is a modular pseudo-Frobenius variety and, in addition, $\Delta(m)$ is the maximum of $S\in\mathscr{T}_{m}$.
	\end{corollary}
	
	In order to build the tree associated with the pseudo-variety $S\in\mathscr{T}_{m}$, we are going to characterise the possible children of each $S\in\mathscr{T}_{m}$.
	
	\begin{proposition}\label{prop27}
		Let $m\in \mathbb{N}\setminus\{0,1\}$, $S\in\mathscr{T}_{m}$, and $x\in\mathrm{msg}(S)\setminus\{m\}$. Then $S\setminus\{x\}\in\mathscr{T}_{m}$ if and only if $\frac{x+m}{2}\notin S$.
	\end{proposition}
	
	\begin{proof}	
		\textit{(Necessity.)} If $\frac{x+m}{2}\in S$, then $\frac{x+m}{2}\in S\setminus\{0,x\}$ and $2\frac{x+m}{2}-m=x\notin S \setminus \{x\}$. Therefore, $S\setminus\{x\}\notin\mathscr{T}_{m}$.
		
		\textit{(Sufficiency.)} If $a\in S\setminus\{0,x\}$, then $2a-m\in S$ because $S\in \mathscr{T}_{m}$. Now, if $2a-m=x$, then $\frac{x+m}{2}=a\in S$. Therefore, $2a-m\not= x$ and, consequently, $2a-m\in S\setminus\{x\}$. Thereby, $S\setminus\{x\}\in \mathscr{T}_{m}$.
	\end{proof}
	
	By applying Theorem~\ref{thm7}, Propositions~\ref{prop25} and \ref{prop27}, and Lemma~\ref{lem10}, we can build the tree $\mathrm{G}(\mathscr{T}_{m})$. Let us see an example.
	
	\begin{example}\label{exmp28}
		In the next figure we have the first four levels of $\mathrm{G}(\mathscr{T}_{4})$.
		
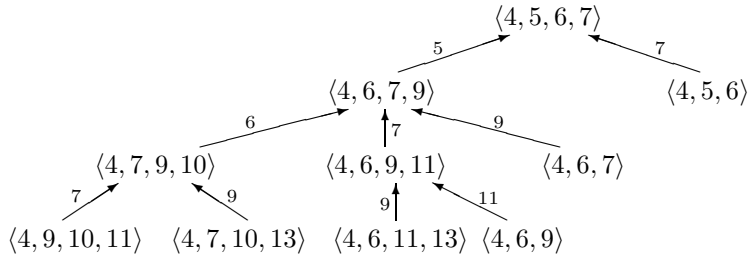
\begin{center}
			\begin{picture}(303,93)
				
				\put(196,84){$\langle 4,5,6,7 \rangle$}
				
				\put(173,73){\scriptsize 5} \put(257,73){\scriptsize 7}
				\put(160,66){\vector(3,1){42}} \put(274,66){\vector(-3,1){42}}
				\put(133,56){$\langle 4,6,7,9 \rangle$} \put(261,56){$\langle 4,5,6 \rangle$}
				
				\put(102,45){\scriptsize 6} \put(157,42){\scriptsize 7} \put(196,45){\scriptsize 9}
				\put(85,38){\vector(4,1){56}} \put(155,38){\vector(0,1){14}} \put(221,38){\vector(-4,1){56}}
				\put(45,28){$\langle 4,7,9,10 \rangle$} \put(132,28){$\langle 4,6,9,11 \rangle$} \put(214,28){$\langle 4,6,7 \rangle$}
				
				\put(36,17){\scriptsize 7} \put(95,17){\scriptsize 9} \put(153,14){\scriptsize 9} \put(190,16){\scriptsize 11}
				\put(33,10){\vector(3,2){21}} \put(103,10){\vector(-3,2){21}} \put(159,10){\vector(0,1){14}} \put(201,10){\vector(-2,1){28}}
				\put(12,0){$\langle 4,9,10,11 \rangle$} \put(74,0){$\langle 4,7,10,13 \rangle$}
				\put(135,0){$\langle 4,6,11,13 \rangle$} \put(191,0){$\langle 4,6,9 \rangle$}
			\end{picture}
		\end{center}
	\end{example}
	
	We finish this section describing the Ap\'ery set for $S\in\mathscr{T}_{m}$.
	
	\begin{proposition}\label{prop29}
		Let $m\in\mathbb{N}\setminus\{0\}$, $S\in\mathscr{T}_{m}$, $\mathrm{msg}(S)=\{m=n_1,n_2,\ldots,n_e\}$. Then $\mathrm{Ap}(S,m)=\{w(0),w(1),\ldots,w(m-1)\}$, where $w(i)$ is the least element of the set $\{a_2n_2+\ldots+a_en_e\mid (a_2,\ldots,a_e)\in\{0,1\}^{e-1}\}$ that is congruent to $i$ modulo $m$.
	\end{proposition}
	
	\begin{corollary}
		Let $m\in\mathbb{N}\setminus\{0\}$ and $S\in\mathscr{T}_{m}$. Then $m = \mathrm{m}(S) \leq 2^{\mathrm{e}(S)-1}$.
	\end{corollary}
	
	\begin{remark}
		Let us observe that we can generalise the concept of thin numerical semigroups in the following way: setting $n\in\mathbb{N}\setminus\{0,1\}$, we say that a numerical semigroup $S$ is $n$-thin if $nx-\mathrm{m}(S)\in S$ for all $x\in S\setminus\{0\}$, and denote by $\mathscr{T}_{n,m}$ the set of all $k$-thin numerical semigroups with multiplicity equal to $m$. It is clear that $\mathscr{T}_{m}=\mathscr{T}_{2,m}$.
		
		Again, $\mathscr{T}_{n,m}$ is a modular pseudo-variety, Proposition~\ref{prop27} is valid for the condition $\frac{x+m}{n}\not\in S$, and we can build the chain
		\[ \mathscr{T}_{2,m} \subsetneq \mathscr{T}_{3,m} \subsetneq \cdots \subsetneq \mathscr{T}_{m,m} = \mathscr{S}_m. \]
		Note that Example~\ref{exmp-01} also gives us the construction that ensures the strict inclusions in this case.
	\end{remark}

	\section{Strong numerical semigroups}\label{strong-ns}
	
	We say that a numerical semigroup $S$ is \textit{strong} if $x+y-\mathrm{m}(S)\in S$ for all $(x,y)\in (S\setminus\{0\})^2$ such that $x\not\equiv y \pmod {\mathrm{m}(S)}$. We denote by $\mathscr{R}_{m}$ the set of all the strong numerical semigroups with multiplicity equal to $m$.
	
	\begin{proposition}\label{prop30}
		Let $S$ be a numerical semigroup with minimal system of generators given by $\{m=n_1<n_2<\cdots<n_e\}$. Then the following two conditions are equivalents.
		\begin{enumerate}
			\item $S\in \mathscr{R}_{m}$.
			\item $\{n_i+n_j-m,3n_i-m\}\subseteq S$ for all $i\in \{2,\ldots,e\}$ and for all $(i,j)\in \{2,\ldots,e\}^2$ such that $i\not= j$.
		\end{enumerate}
	\end{proposition}
	
	\begin{proof}
		\textit{(1. $\Rightarrow$ 2.)} It is enough to observe that $n_i\not\equiv n_j \pmod m$ and that $2n_i\not\equiv n_i \pmod m$.
		
		\textit{(2. $\Rightarrow$ 1.)} Let $x,y \in S\setminus\{0\}$ such that $x\not\equiv y \pmod {\mathrm{m}}$. If $x\equiv 0\pmod m$ or $y\equiv 0\pmod m$, then it is clear that $x+y-m\in S$. Now, if $x\not\equiv 0\pmod m$ and $y\not\equiv 0\pmod m$, then there exists $(i,j)\in \{1,\ldots,m-1\}^2$, with $i\not=j$, and there exists $(p,q)\in\mathbb{N}^2$ such that $x=w(i)+pm$ and $y=w(j)+qm$. Moreover, if there exists $(a,b)\in\{2,\ldots,e\}^2$ such that $a\not=b$ and $w(i)-n_a,w(j)-n_b\in S$, then it is easy to see that $x+y-m\in S$. In other case, there exists $(r,t)\in (\mathbb{N}\setminus \{0\})^2$ such that $w(i)=r\cdot n_a$ and $w(j)=t\cdot n_a$ for some $a\in\{2,\ldots,e\}$. Then, since $i\not=j$, we deduce that $r+t\geq3$ and, consequently, $x+y-m\in S$. In conclusion, $S\in \mathscr{R}_{m}$.
	\end{proof}
	
	The above proposition allows us to easily decide whether a numerical semigroup is strong or not. 
	
	\begin{example}\label{exmp31}
		If $S=\langle 4,5,7 \rangle$, then $\{5+7-4,3\cdot5-4,3\cdot7-4\} = \{8,11,17\} \subseteq S$. Therefore, by Proposition~\ref{prop30}, we have that $S\in \mathscr{R}_{m}$.
	\end{example}
	
	Now we want to show that $\mathscr{R}_{m}$ is a modular Frobenius pseudo-variety. Let us denote by $A=\{1,\ldots,m-1\}^2\setminus\{(1,1),(2,2),\ldots,(m-1,m-1)\}$.
	
	\begin{proposition}\label{prop32}
		$\mathscr{R}_{m}=\mathscr{C}(m,A)$ for all $m\in \mathbb{N}\setminus\{0,1\}$.
	\end{proposition}
	
	\begin{proof}
		Let $S\in\mathscr{R}_{m}$, $\mathrm{Ap}(S,m)=\{w(0),w(1),\ldots,w(m-1)\}$, and $(i,j)\in A$. Then $\{(w(i),w(j)\}\in (S\setminus\{0\})^2$ and $w(i)\not\equiv w(j) \pmod m$. Therefore, $w(i)-w(j)-m\in S$ and, by applying Proposition~\ref{prop1}, we have that $S\in\mathscr{C}(m,A)$.
		
		To see the other inclusion, let $S\in\mathscr{C}(m,A)$ and $(x,y)\in(S\setminus\{0\})^2$. On the one hand, if $0 \in\{x \bmod m, y\bmod m\}$, then it is clear that $x+y-m\in S$. On the other hand, if $0 \notin\{x \bmod m, y\bmod m\}$, then there exist $(i,j)\in A$ and $(p,q)\in \mathbb{N}^2$ such that $x=w(i)+pm$ and $y=w(j)+qm$. Therefore, $x+y-m=(w(i)+w(i)-m)+ (p+q)m \in S$ and, consequently, $S\in\mathscr{R}_{m}$.
	\end{proof}
	
	From Propositions~\ref{prop6} and \ref{prop32}, we get the following result.
	
	\begin{corollary}\label{cor33}
		Let $m\in \mathbb{N}\setminus\{0,1\}$. Then $\mathscr{R}_{m}$ is a modular pseudo-Frobenius variety and, in addition, $\Delta(m)$ is the maximum of $S\in\mathscr{R}_{m}$.
	\end{corollary}
	
	We are now interested in the description of the tree associated with the pseudo-variety $S\in\mathscr{R}_{m}$. In order to do that, we are going to characterise the children of an arbitrary $S\in\mathscr{R}_{m}$.
	
	\begin{proposition}\label{prop34}
		Let $m\in \mathbb{N}\setminus\{0,1\}$, $S\in\mathscr{R}_{m}$, and $x\in\mathrm{msg}(S)\setminus\{m\}$. Then $S\setminus\{x\}\in\mathscr{R}_{m}$ if and only if $x+m\notin \{a+b \mid a,b\in \mathrm{msg}(S)\setminus\{m,x\}, \; a\not= b\} \cup \{3a \mid a \in \mathrm{msg}(S)\setminus\{m,x\}\}$.
	\end{proposition}
	
	\begin{proof}	
		\textit{(Necessity.)} If $a,b\in \mathrm{msg}(S)\setminus\{m,x\}$ and $a\not= b$, then $(a,b) \in (S\setminus\{0,x\})^2$ and $a\not\equiv b \pmod m$. Since $S\setminus\{x\}\in\mathscr{R}_{m}$, we have that $a+b-m\in S\setminus\{x\}$ and, therefore, $a+b-x\not=x$. Thus, $x+m\notin \{a+b \mid a,b\in \mathrm{msg}(S)\setminus\{m,x\}, \; a\not= b\}$.
		
		On the other hand, if $a\in \mathrm{msg}(S)\setminus\{m,x\}$, then $(a,2a)\in (S\setminus\{0,x\})^2$ and $a\not\equiv 2a \pmod m$. Once again, since $S\setminus\{x\}\in\mathscr{R}_{m}$, we have that $3a-m\in S\setminus\{x\}$ and, therefore, $a+b-x\not=x$. Thus, $x+m\notin \{3a \mid a\in \mathrm{msg}(S)\setminus\{m,x\}\}$.
		
		\textit{(Sufficiency.)} Let $a,b\in S\setminus\{0,x\}$ such that $a\not=b$. Since $S\in \mathscr{R}_{m}$, we have that $a+b-m\in S$ and $3a-m\in S$. Now, by Lemma~\ref{lem10}, we know that $\mathrm{msg}(S)\setminus\{x\} \subseteq \mathrm{msg}(S\setminus\{x\}) \subseteq (\mathrm{msg}(S)\setminus\{x\}) \cup \{x+m\}$. Thus, from this fact and the hypothesis, it easily follows that $a+b-m\not=x$ and $3a-m\not=x$, that is, $a+b-m,3a-m\in S\setminus\{x\}$. By applying Proposition~\ref{prop30}, we conclude that $S\setminus\{x\}\in \mathscr{R}_{m}$.
	\end{proof}
	
	From Theorem~\ref{thm7}, Propositions~\ref{prop32} and \ref{prop34}, and Lemma~\ref{lem10}, we can build the tree $\mathrm{G}(\mathscr{R}_{m})$. Let us see an example.
	
	\begin{example}\label{exmp35}
		In the next figure we have the first four levels of $\mathrm{G}(\mathscr{R}_{4})$.
		
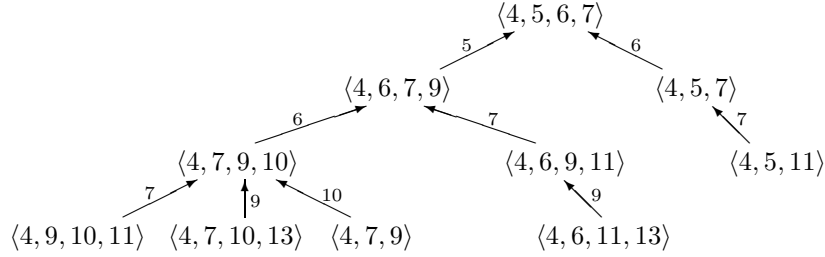
\begin{center}
			\begin{picture}(320,93)
				
				\put(191,84){$\langle 4,5,6,7 \rangle$}
				
				\put(178,73){\scriptsize 5} \put(242,73){\scriptsize 6}
				\put(170,66){\vector(2,1){28}} \put(254,66){\vector(-2,1){28}}
				\put(133,56){$\langle 4,6,7,9 \rangle$} \put(251,56){$\langle 4,5,7 \rangle$}
				
				\put(114,45){\scriptsize 6} \put(188,45){\scriptsize 7} \put(282,45){\scriptsize 7}
				\put(100,38){\vector(3,1){42}} \put(206,38){\vector(-3,1){42}} \put(287,38){\vector(-1,1){14}}
				\put(70,28){$\langle 4,7,9,10 \rangle$} \put(194,28){$\langle 4,6,9,11 \rangle$} \put(279,28){$\langle 4,5,11 \rangle$}
				
				\put(58,17){\scriptsize 7} \put(98,14){\scriptsize 9} \put(125,16){\scriptsize 10} \put(227,16){\scriptsize 9}
				\put(50,10){\vector(2,1){28}} \put(96,10){\vector(0,1){14}} \put(136,10){\vector(-2,1){28}} \put(231,10){\vector(-1,1){14}}
				\put(7,0){$\langle 4,9,10,11 \rangle$} \put(67,0){$\langle 4,7,10,13 \rangle$}
				\put(128,0){$\langle 4,7,9 \rangle$} \put(206,0){$\langle 4,6,11,13 \rangle$}
			\end{picture}		
		\end{center}
	\end{example}
	
	We finish this section describing the Ap\'ery set for $S\in\mathscr{R}_{m}$.
	
	\begin{proposition}\label{prop36}
		Let $m\in\mathbb{N}\setminus\{0,1\}$, $S\in\mathscr{R}_{m}$, $\mathrm{msg}(S)=\{m=n_1,n_2,\ldots,n_e\}$. Then $\mathrm{Ap}(S)=\{w(0),w(1),\ldots,w(n)\}$, where $w(i)$ is the least element of the set $\{0,n_2,\ldots,n_e\} \cup \{2n_2,\ldots,2n_e\}$ that is congruent to $i$ modulo $m$.
	\end{proposition}
	
	\begin{corollary}\label{cor36}
		Let $m\in\mathbb{N}\setminus\{0,1\}$ and $S\in \mathscr{R}_{m}$. Then $m = \mathrm{m}(S) \leq 2\mathrm{e}(S)-1$.
	\end{corollary}
	
	It is interesting to observe that, under the hypotheses of Proposition~\ref{prop36}, $n_i\in\mathrm{Maximals}_{\leq_S}(\mathrm{Ap}(S,m))$ if and only if $2n_i\notin\mathrm{Ap}(S,m)$. In addition, $2n_i\in\mathrm{Ap}(S,m)$ if and only if $2n_i\in\mathrm{Maximals}_{\leq_S}(\mathrm{Ap}(S,m))$. Therefore, from Propositions~\ref{prop21} and \ref{prop36}, we get the next result.
	
	\begin{corollary}\label{cor37}
		Let $m \in \mathbb{N}\setminus \{0,1\}$ and $S\in \mathscr{R}_{m}$. Then $\mathrm{t}(S)=\mathrm{e}(S)-1$.
	\end{corollary}
	
	It is well known that, if $S$ is a numerical semigroup, then $\mathrm{e}(S) \leq \mathrm{m}(S)$ and $\mathrm{t}(S) \leq \mathrm{m}(S)-1$ (see Proposition~2.10 and Corollary~2.23 in \cite{springer}). Combining these facts with Corollaries~\ref{cor36} and \ref{cor37}, we have the next result.
	
	\begin{corollary}\label{cor38}
		Let $m \in \mathbb{N}\setminus \{0,1\}$ and $S\in \mathscr{R}_{m}$. Then $\frac{\mathrm{m}(S)-1}{2} \leq \mathrm{t}(S) \leq \mathrm{m}(S)-1$ (or, equivalently, $\frac{\mathrm{m}(S)+1}{2} \leq \mathrm{e}(S) \leq \mathrm{m}(S)$).
	\end{corollary}
	
	\begin{remark}
		Contrary to what happens in Sections~\ref{2-level-ns} and \ref{thin-ns}, the concept of strong numerical semigroup does not have a natural generalisation. In fact, we have, at least, two possibilities.
		\begin{enumerate}
			\item $x_1+\cdots x_n-\mathrm{m}(S)\in S$ for all $(x_1,\ldots,x_n) \in (S\setminus\{0\})^n$ such that $x_i\not\equiv x_j \pmod {\mathrm{m}(S)}$, for all $i\not=j$. In this case, $(x_1,\ldots,x_n) \in A_n$, where
			\[ A_n=\left\{ (\alpha_1,\ldots,\alpha_n) \in \{1,\ldots,m-1\}^n \mid \alpha_i\not=\alpha_j \mbox{ for all } i\not=j \right\}. \]
			\item $x_1+\cdots x_n-\mathrm{m}(S)\in S$ for all $(x_1,\ldots,x_n) \in B_n$, where
			\[ B_n=\{1,\ldots,m-1\}^n \setminus \{(1,\ldots,1),(2,\ldots,2),\ldots,(m-1,\ldots,m-1)\}.\]
		\end{enumerate}
		It is clear that, if $n\geq3$, then $A_n\subsetneq B_n$ and, in consequence, $\mathscr{C}(m,B_n) \subseteq \mathscr{C}(m,A_n)$. In addition, it is guessed that, if $n<m$, then the second inclusion will be strict.
	\end{remark}

	\section*{Acknowledgement}
%
	
	This version of the article has been accepted for publication, after peer review but is not the Version of Record and does not reflect post-acceptance improvements, or any corrections. The Version of Record is available online at: http://dx.doi.org/10.1007/s13348-021-00339-0.

\end{document}